\documentclass[reqno,12pt,letterpaper]{amsart}
\usepackage{amssymb,oldlfont}
\usepackage{amsmath}
\usepackage{mathrsfs}
\usepackage{hyperref}
\usepackage{color}
\def\wrtext#1{\relax\ifmmode{\leavevmode\hbox{#1}}\else{#1}\fi}
\def\abs#1{\left|#1\right|}
\def\begeq{\begin{equation}}
	\def\endeq{\end{equation}}

\let\epsilon=\varepsilon

\def\part#1{\frac{\partial}{\partial #1}}

\def\norm#1{||\,#1\,||}

\newcommand{\OO}{{\mathcal O}}
\newcommand{\RR}{{\mathbb R}}

\newcommand{\GG}{{\mathcal G}}
\newcommand{\NN}{{\mathbb N}}
\newcommand{\CC}{{\mathbb C}}

\newcommand{\xbar}{\overline{x}}

\newcommand{\ptilde}{\widetilde{p}}

\newcommand{\CI}{{{C}^\infty}}
\newcommand{\CIb}{{{C}^\infty_b}}

\let\Im=\Imag

\DeclareMathOperator{\sgn}{sgn}

\let\Re=\Real

\newcommand{\real}{{\mathbb R}}
\newcommand{\comp}{{\mathbb C}}

\newcommand{\nat}{{\mathbb N}}

\newcommand{\dist}{\mbox{\rm dist\,}}

\renewcommand{\Re}{\mbox{\rm Re\,}}
\renewcommand{\Im}{\mbox{\rm Im\,}}
\renewcommand{\exp}{\mbox{\rm exp\,}}
\newcommand{\supp}{\mbox{\rm supp}}

\def\Re{{\rm Re\,}}
\def\Im{{\rm Im\,}}

\newtheorem{dref}{Definition}
\newtheorem{theo}{Theorem}
\newtheorem{lemma}[dref]{Lemma}
\newtheorem{prop}[dref]{Proposition}
\newtheorem{remark}[dref]{Remark}

\numberwithin{equation}{section}
\numberwithin{dref}{section}

\title{Boundary spectral estimates for semiclassical Gevrey operators}
\author{Haoren Xiong}
\email{haorenxiong@math.ucla.edu}
\address{Department of Mathematics, University of California,
	Los Angeles, CA 90095, USA}

\begin{document}
	
	\begin{abstract}
		We obtain the spectral and resolvent estimates for semiclassical pseudodifferential operators with symbol of Gevrey-$s$ regularity, near the boundary of the range of the principal symbol. We prove that  the boundary spectrum free region is of size ${\mathcal O}(h^{1-\frac{1}{s}})$ where the resolvent is at most fractional exponentially large in $h$, as the semiclassical parameter $h\to 0^+$. This is a natural Gevrey analogue of a result by N. Dencker, J. Sj{\"o}strand, and M. Zworski in the $C^{\infty}$ and analytic cases.
	\end{abstract}
	
	\maketitle
	
	\section{Introduction and statement of results}
	\label{sec:intro}
	
	In this work, we study the spectrum of a non-self-adjoint semiclassical Gevrey pseudodifferential operator, as the semiclassical parameter $h\to 0^+$. Unlike self-adjoint operators, it is well known that the spectrum of a non-self-adjoint operator may lie deep inside the range of its leading symbol as $h\to 0^+$. For instance, the complex harmonic oscillator: $-h^2\frac{d^2}{dx^2} + i x^2$ on $L^2(\RR)$, which was used by Davies \cite{Davies} as an inspiring example of non-normal differential operators, has purely discrete spectrum $\{e^{i\pi/4} h (2k+1) : k\in\NN\}$; while the range of its symbol $\xi^2 + i x^2$ on $(x,\xi)\in T^*\RR$ is the sector $\{z\in\CC : 0\leq \arg z\leq \pi/2\}$.
	
	\noindent
	More generally, Dencker, Sj{\"o}strand, and Zworski \cite{DeSjZw} considered spectral estimates for quantizations of bounded functions, with all derivatives bounded,
	\[
	p\in\CIb(\RR^{2n}):=\{u\in\CI(\RR^{2n}) : \partial^\alpha u\in L^\infty(\RR^{2n}),\ \forall\alpha\in\NN^{2n}\}.
	\]
	It has been pointed out in \cite{DeSjZw} that the case of functions whose values avoid a point in $\CC$ and tend to infinity as $(x,\xi)\to\infty$ can be reduced to this case. Let us denote the closure of range of $p$ by $\Sigma(p):= \overline{p(T^* \RR^n)}$,
	and denote by $\Sigma_\infty (p)$ the set of accumulation points of $p$ at infinity. Let $z_0\in\partial\Sigma(p)$ and suppose the principal-type condition:
	\begin{equation}
		\label{principal type condition}
		dp (x,\xi)\neq 0,\quad\text{if }\, p(x,\xi) = z_0,\ \,(x,\xi)\in T^*\RR^n.
	\end{equation}
	For every $\rho\in p^{-1}(z_0)$ with $z_0\in\partial\Sigma(p)$, let $\theta=\theta(\rho)\in\RR$ be such that $e^{-i\theta} d p$ is real at $\rho$. Let us also assume a nontrapping condition on $p$ and $z_0$ (see \cite{DeSjZw} or \cite{Sj10}): 
	\begin{equation}
		\label{dynamical condition}
		\begin{gathered}
			\text{For every $\rho\in p^{-1}(z_0)$, the complete trajectory of $H_{\Re(e^{-i\theta(\rho)} p)}$} \\
			\text{that passes through}\ \rho\ \text{is not contained in}\ p^{-1}(z_0) , 
		\end{gathered}
	\end{equation}
	where $H_f(\rho) = (f_\xi'(\rho),-f_x'(\rho))$, $\rho\in T^*\RR^{n}$, is the Hamilton vector field of $f$.
	
	\noindent
	Under these conditions it has been proved in \cite{DeSjZw} that for a semiclassical operator $P(h)$ whose principal part is given by the Weyl quantization of $p\in\CIb(\RR^{2n})$:
	\begin{equation}
		\label{eq:Weyl quantization}
		p^w(x,hD) u(x) =  \frac{1}{(2\pi h)^n}\int_{\RR^{2n}} e^{\frac{i}{h}(x-y)\cdot\theta} p\left(\frac{x+y}{2},\theta\right) u(y)\,dy \,d\theta ,\quad u\in{\mathscr S}(\RR^n),
	\end{equation}
	if $z_0\in \partial\Sigma(p)\setminus \Sigma_\infty(p)$ satisfies \eqref{principal type condition} and \eqref{dynamical condition}, then for any $M>0$ there exists $h_0(M)>0$ such that
	\begin{equation}
		\label{DSZ Thm 1.3 smooth}
		\{z\in\CC : |z-z_0|<M h \log(1/h)\} \cap \sigma(P(h)) = \emptyset,\quad 0<h<h_0(M),
	\end{equation}
	where $\sigma(P(h))$ denotes the spectrum of $P(h)$.
	
	\noindent
	Furthermore, in the case where $p$ is a bounded holomorphic function in a tubular neighborhood of $\RR^{2n}\subset\CC^{2n}$, then there exist $\delta_0,\, h_0>0$ such that 
	\begin{equation}
		\label{DSZ Thm 1.3 analytic}
		\{z\in\CC : |z-z_0|< \delta_0\} \cap \sigma(P(h)) = \emptyset,\quad 0<h<h_0 .
	\end{equation}
	
	\noindent
	If we compare the size of the spectrum free region near the boundary of $\Sigma(p)$ for bounded smooth symbols $p$ in \eqref{DSZ Thm 1.3 smooth} with that for bounded holomorphic symbols in \eqref{DSZ Thm 1.3 analytic}, we observe that the size improves from $\OO(h\log(1/h))$ to $\OO(1)$. Motivated by that, the purpose of this work is to explore the spectrum free region near the boundary of $\Sigma(p)$ for bounded Gevrey symbols $p$ (see the definition below), which can be viewed as an interpolating case between the bounded smooth and holomorphic symbols.
	
	\medskip
	\noindent
	The consideration of Gevrey (pseudo)differential operators has a long-standing tradition in the theory of PDEs, beginning with the seminal work \cite{BoKr}, see also \cite{Rodino}, \cite{Lascar1988}, \cite{Lascar1997}. Gevrey regularity problems have been studied in various contexts, including scattering theory \cite{Rouleux}, \cite{GaZw21}, the FBI transform in Gevrey classes \cite{GuedesFBIGevrey}, the propagation of Gevrey singularities in boundary value problems \cite{Lebeau}, \cite{Lascar1991}, and Gevrey pseudodifferential operators in the complex domain \cite{HiLaSjZeI, HiLaSjZeII}.
	
	\medskip
	\noindent
	Let $s> 1$. The bounded (global) Gevrey-$s$ class, denoted by $\GG_{\rm b}^s(\RR^d)$, consists of all functions $u\in C^\infty(\RR^d)$ such that there exists $C>0$ such that
	\begin{equation*}
		\label{Gevrey definition}
		|\partial^\alpha u(x) | \leq C^{1+|\alpha|} \alpha!^s,\quad \forall\alpha\in\NN^d,\ x\in \RR^d.
	\end{equation*}
	The Denjoy–-Carleman theorem \cite[Theorem 1.3.8]{hormander1985analysis} implies that when $s>1$ the Gevrey-$s$ class is non-quasianalytic, i.e.
	$\GG_{\rm c}^s(\RR^d) := \GG_{\rm b}^s(\RR^d)\cap C_{\rm c}^\infty(\RR^n) \neq \{0\}$. Therefore, there are $\GG^s$ partitions of unity.
	
	\noindent
	To study semiclassical Gevrey operators, we define Gevrey symbols as functions in $\GG_{\rm b}^s(\RR^{2n})$ that may depend on the semiclassical parameter $h\in(0,1]$. More precisely, we write $a(\cdot;h)\in \GG_{\rm b}^s(\RR^{2n})$ if and only if for some $C>0$ uniformly in $h\in(0,1]$ we have
	\begin{equation}
		\label{eq:Gevrey symbol}
		|\partial_x^\alpha \partial_\theta^\beta a(x,\theta;h)|\leq C^{1+|\alpha|+|\beta|} \alpha!^s \beta!^s , \quad \forall\alpha, \beta\in\NN^n,\ (x,\theta)\in \RR^{2n}.
	\end{equation}
	Let us then introduce semiclassical Gevrey pseudo-differential operators, which are semiclassical Weyl quantizations of $a(\cdot;h)\in\GG_{\rm b}^s(\RR^{2n})$ acting on $u\in {\mathscr S}(\RR^n)$,  
	\begin{equation*}
		a^w(x,hD;h)u(x) = \frac{1}{(2\pi h)^n}\int_{\RR^{2n}} e^{\frac{i}{h}(x-y)\cdot\theta} a\left(\frac{x+y}{2},\theta;h\right) u(y)\,dy \,d\theta .
	\end{equation*}
	We recall that $a^w(x,hD;h)$ extends to a bounded operator on $L^2(\RR^d)$ uniformly in $h\in(0,1]$, see for instance \cite[Section 4.5]{zw2012}. 
	
	\noindent
	We say that an $h$-independent function $a_0\in\GG_{\rm b}^s(\RR^{2n})$ is the principal symbol of the semiclassical symbol $a(\cdot;h)\in \GG_{\rm b}^s(\RR^{2n})$ if there exists $r(\cdot;h)\in\GG_{\rm b}^s(\RR^{2n})$ such that
	\[
	a(x,\theta;h) = a_0(x,\theta) + h r(x,\theta;h),\quad \forall (x,\theta)\in \RR^{2n},\ h\in(0,1] .
	\]
	Moreover, $a_0^w(x,hD)$ is called the principal part of $a^w(x,hD;h)$.
	
	\medskip
	\noindent
	The following is the main result of this work. 
	\begin{theo}
		\label{main theorem}
		Let $p\in\GG^s_b(\RR^{2n})$, $s>1$, and let $p^w(x,hD)$ be the principal part of a semiclassical Gevrey operator $P(h)=P(x,hD;h)$. If $p$ and $z_0\in \partial\Sigma(p)\setminus \Sigma_\infty(p)$ satisfy conditions \eqref{principal type condition} and \eqref{dynamical condition}, then there exist $h_0>0$ and $C>0$ such that
		\[
		\left\{z\in\comp : |z-z_0| < C^{-1} h^{1-\frac{1}{s}} \right\} \cap \sigma(P(h)) = \emptyset,\quad 0<h<h_0 .
		\]
		Furthermore, for $z\in\comp$ with $|z-z_0| < C^{-1} h^{1-\frac{1}{s}}$ we have the resolvent estimate:
		\begin{equation}
			\label{resolvent estimate}
			(P(h)-z)^{-1} = \OO(1)\exp(\OO(1)h^{-\frac{1}{s}}) : L^2(\RR^n)\to L^2(\RR^n),\quad 0<h<h_0 .
		\end{equation}
	\end{theo}
	
	\begin{remark}
		In the context of resonances, it has been shown in \cite{Rouleux1998,Rouleux} that semiclassical Schr\"odinger operators with $\GG^s$ potentials which are dilation analytic near infinity have a resonance free region of size $\OO(h^{1-\frac{1}{s}})$ near a non-trapping energy level in the semiclassical limit $h\to 0^+$ and that the exponent $1-1/s$ is optimal by constructing a $\GG^s$ potential such that there exist resonances $E$ near a non-trapping energy level $E_0>0$ with $\Im E\approx -C h^{1-\frac{1}{s}}$, $C>0$, for $h$ sufficiently small \cite{Rouleux1998}. We can therefore infer that the exponent $1-1/s$ for the spectrum free region in Theorem \ref{main theorem} is optimal. 
	\end{remark}
	
	\noindent
	Let us also highlight a special case of Theorem \ref{main theorem}, which is commonly considered for evolution equations or semigroups $\exp(-tP/h)$, i.e. the case where
	\begin{equation}
		\label{special case}
		z_0\in i\RR,\quad \Re p \geq 0\quad \textrm{near}\ \, p^{-1}(z_0).
	\end{equation}
	The principal-type condition \eqref{principal type condition} in this case implies that
	\begin{equation}
		\label{principal type special case}
		d\Im p \neq 0,\quad d\Re p = 0\quad \textrm{on}\ \, p^{-1}(z_0).
	\end{equation}
	In view of the nontrapping condition \eqref{dynamical condition}, we assume in this case that
	\begin{equation}
		\label{nontrapping special case}
		\begin{gathered}
			\forall \rho\in p^{-1}(z_0),\quad \textrm{the maximal trajectory of $H_{\Im p}$ passing through $\rho$} \\
			\textrm{contains a point where $\Re p>0$} .
		\end{gathered}
	\end{equation}
	Under these conditions, we have the following:
	\begin{theo}
		\label{thm:special case}
		Let $p\in\GG^s_b(\RR^{2n})$, $s>1$, and let $p^w(x,hD)$ be the principal part of a semiclassical Gevrey operator $P(h)=P(x,hD;h)$. Suppose that $p$ and $z_0\in \partial\Sigma(p)\setminus \Sigma_\infty(p)$ satisfy conditions \eqref{special case}--\eqref{nontrapping special case}. Then there exist $h_0>0$ and $C>0$ such that 
		\[
		\left\{z\in\comp : |z-z_0| < C^{-1} h^{1-\frac{1}{s}} \right\} \cap \sigma(P(h)) = \emptyset,\quad 0<h<h_0 ,
		\]
		and that the estimate \eqref{resolvent estimate} holds for $z\in\comp$ with $|z-z_0| < C^{-1} h^{1-\frac{1}{s}}$.
	\end{theo}
	
	\noindent
	In the case of analytic symbols, Dencker et al \cite{DeSjZw} proved \eqref{DSZ Thm 1.3 analytic} by studying the action of $P(h)$ on microlocally weighted spaces associated to a family of complex distorted {\it IR manifolds} via the FBI transform, see \cite{HelfflerSj} for the original method and \cite[Chapter 12]{Sj02}, \cite{M_book} for a detailed presentation. It is worth recalling that the key ingredient in the proof of \eqref{DSZ Thm 1.3 analytic} in \cite{DeSjZw} is a Toeplitz identity that connects the action of semiclassical pseudodifferential operators on the microlocally weighted spaces to the multiplication by the principal symbols. Such a result is essentially well known, see \cite{Sj90} for the analytic case, \cite{HeSjSt} for the smooth case, and \cite{Rouleux} for the Gevrey case. Thanks to the techniques recently developed in \cite{HiLaSjZeI}, we will use a straightforward argument to establish the Toeplitz identity in the Gevrey setting, see Section \ref{sec:review Gevrey op}. 
	
	\medskip
	\noindent
	The paper is organized as follows. In section \ref{sec:review Gevrey op}, we review and introduce some essential tools for semiclassical pseudodifferential operators with Gevrey symbols, including a Toeplitz identity in the complex domain and a composition formula in the real domain. Section \ref{sec:complex deformation} is devoted to the proof of Theorem \ref{thm:special case} by introducing small complex deformations of $\RR^{2n}$ and working on the FBI transform side. In Section \ref{sec: prove Thm 1}, we introduce a Gevrey multiplier using a version of the Malgrange preparation theorem for Gevrey functions, which allows us to reduce Theorem \ref{main theorem} to the more special Theorem \ref{thm:special case}, thus completing the proof of our main theorem.
	
	\medskip
	
	\noindent
	{\sc Acknowledgments.} The author would like to thank Michael Hitrik for many valuable insights and helpful discussions and for encouraging us to pursue the Gevrey case. The author would also like to thank Alix Deleporte for pointing out the reference \cite{Br}.
	
	\section{Review of semiclassical Gevrey operators in the complex domain}
	\label{sec:review Gevrey op}

	Let $\Phi_0$ be a strictly plurisubharmonic quadratic form on $\comp^n$ and let us set
	\begin{equation}
		\label{eq1.1}
		\Lambda_{\Phi_0} = \left\{\left(x,\frac{2}{i}\frac{\partial \Phi_0}{\partial x}(x)\right), \, x\in \comp^n\right\} \subset \comp^{2n}.
	\end{equation}
	Let us also introduce the Bargmann space
	\begin{equation}
		\label{eq1.2}
		H_{\Phi_0}(\comp^n) = {\rm Hol}(\comp^n) \cap L^2(\comp^n, e^{-2\Phi_0/h} L(dx)),
	\end{equation}
	where $L(dx)$ is the Lebesgue measure on $\comp^n$ and $0 < h \leq 1$ is the semiclassical parameter. Using the projection map $\pi_x : \Lambda_{\Phi_0}\owns(x,\xi)\mapsto x\in\comp_x^n\cong\RR^{2n}$, we identify $\Lambda_{\Phi_0}$ with $\comp_x^n$ and define the Gevrey spaces $\GG_{\rm b}^s(\Lambda_{\Phi_0})$, $\GG_{\rm c}^s(\Lambda_{\Phi_0})$. Let $a\in {\cal G}^s_b(\Lambda_{\Phi_0})$ be an $h$-independent symbol, for some $s>1$, and let $u\in {\rm Hol}(\comp^n)$ be such that
	$$
	u(x) = {\cal O}_{h,N}(1) \langle{x\rangle}^{-N} e^{\Phi_0(x)/h},
	$$
	for all $N\in \nat$. We introduce the semiclassical Weyl quantization of $a$ acting on $u$,
	\begin{equation}
		\label{eq1.3}
		a^w_{\Gamma}(x,hD_x) u(x) = \frac{1}{(2\pi h)^n}\int\!\!\!\!\int_{\Gamma(x)} e^{\frac{i}{h}(x-y)\cdot \theta} a\left(\frac{x+y}{2},\theta\right)u(y)\, dy\wedge d\theta.
	\end{equation}
	Here $\Gamma(x)\subset \comp^{2n}_{y,\theta}$ is the natural integration contour given by
	\begin{equation}
		\label{eq1.4}
		\theta = \frac{2}{i} \frac{\partial \Phi_0}{\partial x}\left(\frac{x+y}{2}\right),\quad y\in \comp^n.
	\end{equation}
	Let next $\Phi_1 \in C^{1,1}(\comp^n;\real)$ be such that
	\begin{equation}
		\label{eq1.5}
		\norm{\nabla^k(\Phi_1 - \Phi_0)}_{L^{\infty}(\comp^n)} \leq C^{-1} h^{1 - \frac{1}{s}},\quad k = 0,1,2,
	\end{equation}
	for some $C>0$ sufficiently large. We set $\displaystyle \omega = h^{1-\frac{1}{s}}$ and introduce the following $2n$--dimensional Lipschitz contour for $x\in \comp^n$,
	\begin{equation}
		\label{eq1.6}
		\Gamma^{\Phi_1}_{\omega}(x): \quad \theta=\frac{2}{i}\frac{\partial \Phi_1}{\partial x}\left(\frac{x+y}{2}\right)
		+ if_\omega(x-y),\quad y\in \comp^n,
	\end{equation}
	where
	\begin{equation}
		\label{eq1.7}
		f_\omega({z})= \begin{cases}
			\hskip10pt\overline{{z}}, \quad \, |{z}|\leq \omega,\\
			{}\\
			\displaystyle \omega |z|^{-1} \overline{{z}}, \quad |{z}| > \omega.
		\end{cases}
	\end{equation}
	
	\medskip
	\noindent
	Let $\widetilde{a}\in {\cal G}^s_b(\comp^{2n})$ be an almost holomorphic extension of $a$ such that ${\rm supp}\, \widetilde{a} \subset \Lambda_{\Phi_0} + B_{{\comp}^{2n}}(0,C_0)$, for some $C_0>0$. We remark that the existence of such an almost holomorphic extension whose Gevrey order is the same as that of $a$ is due to Carleson \cite{carleson1961universal} (see also \cite{GuedesFBIGevrey}).
	It has been established in Theorem 1.1 of~\cite{HiLaSjZeI} that for $1<s\leq 2$ (the complementary range $s>2$ will be discussed later), 
	\begin{equation}
		\label{eq1.8}
		\begin{gathered}
			a^w_{\Gamma}(x,hD_x) - \widetilde{a}^w_{\Gamma^{\Phi_1}_{\omega}}(x,hD_x) = {\cal O}(1)\, \exp\left(- C^{-1} h^{-\frac{1}{s}}\right):\\
			H_{\Phi_1}(\comp^n) \rightarrow L^2(\comp^n, e^{-2\Phi_1/h} L(dx)),
		\end{gathered}
	\end{equation}
	where $C>0$ is a constant and we have set, similarly to {\rm (\ref{eq1.2})},
	\[
	H_{\Phi_1}(\comp^n) = {\rm Hol}(\comp^n) \cap L^2(\comp^n, e^{-2\Phi_1/h} L(dx)).
	\]
	The realization
	\begin{equation}
		\label{eq1.81}
		\widetilde{a}^w_{\Gamma^{\Phi_1}_{\omega}}(x,hD_x)u(x) = \frac{1}{(2\pi h)^n}\int\!\!\!\!\int_{\Gamma^{\Phi_1}_{\omega}(x)} e^{\frac{i}{h}(x-y)\cdot \theta}\, \widetilde{a}\left(\frac{x+y}{2},\theta\right)u(y)\, dy\wedge d\theta
	\end{equation}
	satisfies
	\begin{equation}
		\label{eq1.9}
		\widetilde{a}^w_{\Gamma^{\Phi_1}_{\omega}}(x,hD_x) = {\cal O}(1): H_{\Phi_1}(\comp^n) \rightarrow L^2(\comp^n, e^{-2\Phi_1/h} L(dx)).
	\end{equation}
	
	\medskip
	\noindent
	For future reference, let us also recall the following version of the Fourier inversion formula in the complex domain, see for instance \cite{HiSj}. Let $u\in {\rm Hol}(\comp^n)$ be such that $u(x) = {\cal O}_{h}(1) \langle{x\rangle}^{N_0} e^{\Phi_0(x)/h}$, for some $N_0 \geq 0$. We have
	\begin{equation}
		\label{eq1.10}
		u(x) = \frac{1}{(2\pi h)^n}\int\!\!\!\!\int_{\Gamma^{\Phi_1}_{\omega}(x)} e^{\frac{i}{h}(x-y)\cdot \theta} u(y)\, dy\wedge d\theta.
	\end{equation}
	Indeed, \eqref{eq1.10} follows, by an application of Stokes' formula, from the corresponding Fourier inversion formula when the contour $\Gamma^{\Phi_1}_{\omega}(x)$ is replaced by the more standard contour of the form
	\[
	\theta = \frac{2}{i} \frac{\partial \Phi_0}{\partial x}\left(\frac{x+y}{2}\right) + i\overline{(x-y)}, \quad y\in \comp^n.
	\]
	In particular, \eqref{eq1.10} holds for $u\in H_{\Phi_1}(\comp^n)= H_{\Phi_0}(\comp^n)$ (they are equal as linear spaces). Similarly, for such functions, we find by Stokes' formula, writing $D_{x_j}:=i^{-1}\partial_{x_j}$,
	\begin{equation}
		\label{eq1.10.1}
		hD_{x_j} u(x) = \frac{1}{(2\pi h)^n}\int\!\!\!\!\int_{\Gamma^{\Phi_1}_{\omega}(x)} e^{\frac{i}{h}(x-y)\cdot \theta} \theta_j u(y)\, dy\wedge d\theta, \quad 1 \leq j \leq n.
	\end{equation}
	
	\medskip
	\noindent
	Setting $\displaystyle \xi_1(x) = \frac{2}{i} \frac{\partial \Phi_1}{\partial x}(x)$, we get by a Taylor expansion, when $(y,\theta) \in \Gamma^{\Phi_1}_{\omega}(x)$,
	\begin{equation}
		\label{eq1.11}
		\widetilde{a}\left(\frac{x+y}{2},\theta\right) = \widetilde{a}(\rho_0) + \partial_x\widetilde{a}(\rho_0)\cdot \Delta y +
		\partial_{\theta}\widetilde{a}(\rho_0)\cdot \Delta\theta + r(x,y,\theta),
	\end{equation}
	where we have adopted more convenient notation: $\rho_0 = (x,\xi_1(x))$, $\Delta y = \frac{y-x}{2}$, $\Delta\theta = \theta-\xi_1(x)$, and we have written 
	\begin{equation}
		\label{eq:r(x,y,theta)}
		r(x,y,\theta) := \partial_{\xbar}\widetilde{a}(\rho_0)\cdot \overline{\Delta y} + \partial_{\overline{\theta}}\widetilde{a}(\rho_0)\cdot \overline{\Delta\theta}  + \int_0^1 (1-t)\, \widetilde{a}_t^{(2)}(x,y,\theta) \,dt	.
	\end{equation}
	Here $\widetilde{a}_t^{(2)}(x,y,\theta)$ is defined by, writing $\rho_t := ( x+ t\Delta y\, ,\, \xi_1(x)+t\Delta\theta )$,
	\begin{multline}
		\widetilde{a}_t^{(2)}(x,y,\theta) = \widetilde{a}''_{xx}(\rho_t){\Delta y}\cdot{\Delta y} + 2\widetilde{a}''_{x\theta}(\rho_t) \Delta\theta\cdot{\Delta y} + \widetilde{a}''_{\theta\theta}(\rho_t)\Delta\theta\cdot \Delta\theta \\
		\quad + 2\widetilde{a}''_{x\xbar}(\rho_t)\overline{{\Delta y}}\cdot{\Delta y} + 2\widetilde{a}''_{x\overline{\theta}}(\rho_t)\overline{\Delta\theta}\cdot{\Delta y} + 2\widetilde{a}''_{\xbar\theta}(\rho_t)\Delta\theta\cdot\overline{{\Delta y}} + 2\widetilde{a}''_{\theta\overline{\theta}}(\rho_t)\overline{\Delta\theta}\cdot\Delta\theta \\
		+ \widetilde{a}''_{\xbar\xbar}(\rho_t)\overline{{\Delta y}}\cdot\overline{{\Delta y}} + 2\widetilde{a}''_{\xbar\overline{\theta}}(\rho_t)\overline{\Delta\theta}\cdot\overline{{\Delta y}} + \widetilde{a}''_{\overline{\theta}\overline{\theta}}(\rho_t)\overline{\Delta\theta}\cdot\overline{\Delta\theta} .\qquad\qquad
	\end{multline}
	We note that for $0\leq t\leq 1$,
	\begin{multline}
		\label{dist Lambda Phi_0}
		\textrm{dist}\big( \rho_t,\Lambda_{\Phi_0}\big) \leq \left|\frac{2}{i}\frac{\partial\Phi_0}{\partial x}\left((1-t)x+t\frac{x+y}{2}\right) - (1-t)\frac{2}{i}\frac{\partial\Phi_1}{\partial x}(x) - t\theta \right| \\
		\quad\quad\quad\quad\leq 2\left|(1-t)\frac{\partial(\Phi_0-\Phi_1)}{\partial x}(x)+ t\frac{\partial(\Phi_0-\Phi_1)}{\partial x}\left(\frac{x+y}{2}\right) \right| + t|f_\omega(x-y)| \\
		\leq 2\|\nabla(\Phi_1-\Phi_0)\|_{L^\infty(\comp^n)} + t\omega \leq \OO(1)h^{1-\frac{1}{s}} ;\qquad\qquad\qquad\qquad\quad\,
	\end{multline}
	We also recall from \cite[Remark 1.7]{GuedesFBIGevrey} that for any $\alpha,\beta,\gamma,\delta\in\NN^n$ there exists $A, C>0$ such that
	\begin{multline}
		\label{eq1.13}
		\abs{\partial_x^\alpha \partial_{\xbar}^\beta \partial_\theta^\gamma \partial_{\overline{\theta}}^\delta \bigl(\overline{\partial} \widetilde{a}(x,\xi)\bigr)} \leq A C^{|\alpha|+|\beta|+|\gamma|+|\delta|} \alpha!^s \beta!^s \gamma!^s \delta!^s \\
		\cdot\exp\left(-C^{-1}\textrm{dist}\big((x,\xi),\Lambda_{\Phi_0}\big)^{-\frac{1}{s-1}}\right),\quad (x,\xi)\in \comp^{2n}.
	\end{multline}
	Combing \eqref{eq:r(x,y,theta)}--\eqref{eq1.13} we conclude
	\begin{equation}
		\label{eq1.12}
		r(x,y,\theta) = {\cal O}(1)\abs{x-y}^2 + {\cal O}(1)\,\exp\bigl(- C^{-1}h^{-\frac{1}{s}}\bigr), \quad C>0.
	\end{equation}
	Here we have also used the fact that along the contour $\Gamma^{\Phi_1}_{\omega}(x)$, we have
	\[
	\abs{\theta - \xi_1(x)} \leq \norm{\nabla^2 \Phi_1}_{L^\infty(\comp^n)} \abs{x-y} + \abs{f_{\omega}(x-y)}  \leq {\cal O}(\abs{x-y}).
	\]
	Let us set
	\begin{equation}
		\label{eq1.14}
		Ru(x) = \frac{1}{(2\pi h)^n}\int\!\!\!\!\int_{\Gamma^{\Phi_1}_{\omega}(x)} e^{\frac{i}{h}(x-y)\cdot \theta}
		r(x,y,\theta)\, u(y) dy\wedge d\theta,
	\end{equation}
	we shall next check that
	\begin{equation}
		\label{eq1.15}
		R = {\cal O}(h): L^2(\comp^n, e^{-2\Phi_1/h} L(dx)) \rightarrow L^2(\comp^n, e^{-2\Phi_1/h} L(dx)).
	\end{equation}
	To this end, we consider the distribution kernel of $R$, writing
	$$
	Ru(x) = \int k(x,y;h) u(y)\, L(dy),
	$$
	we infer from the proof of Theorem 3.3 in~\cite{HiLaSjZeI} together with (\ref{eq1.12}) that 
	\begin{equation}
		\label{eq1.16}
		e^{-\frac{\Phi_1(x)}{h}} \abs{k(x,y;h)} e^{\frac{\Phi_1(y)}{h}} \leq
		{\cal O}(1) h^{-n} \bigl(\abs{x-y}^2 + e^{-C^{-1}h^{-\frac{1}{s}}}\bigr) e^{-\frac{F_{\omega}(x-y)}{2h}} 
	\end{equation}
	provided that the constant in (\ref{eq1.5}) is sufficiently large. Here, following~\cite{HiLaSjZeI}, we have set
	\begin{equation}
		\label{eq1.17}
		0 \leq F_\omega({z})= {\rm Re}\, (z\cdot f_{\omega}(z)) = \begin{cases}
			\vert{{z}}\vert^2, \quad \, |{z}|\leq \omega, \\
			{}\\
			\omega\vert{z}\vert, \quad\, |{z}| > \omega.
		\end{cases}
	\end{equation}
	In view of Schur's lemma, we only have to control the $L^1$ norm
	\begin{multline}
		\label{eq1.18}
		h^{-n}\int |x|^2 e^{-\frac{F_{\omega}(x)}{2h}} L(dx) = h^{-n}\int_{\abs{x}\leq \omega} |x|^2 e^{-\frac{|x|^2}{2h}} L(dx) + h^{-n}\int_{\abs{x}\leq \omega} |x|^2 e^{-\frac{\omega|x|}{2h}} L(dx) \\
		\leq \OO(1)h + \OO(1)h \frac{h^{n+1}}{\omega^{2n+2}} = \OO(h),\qquad\qquad\qquad\qquad
	\end{multline}
	since $\dfrac{h}{\omega^2} = h^{\frac{2}{s} -1}\leq 1$ if $1 < s \leq 2$. The estimate (\ref{eq1.15}) therefore follows, and combining it with (\ref{eq1.8}), (\ref{eq1.11}), (\ref{eq1.10}), and (\ref{eq1.10.1}), we get for $u\in H_{\Phi_1}(\comp^n)$,
	\begin{equation}
		\label{eq1.19}
		a^w(x,hD_x)u(x) = \widetilde{a}(x,\xi_1(x))u(x) + \partial_{\theta}\, \widetilde{a}(x,\xi_1(x))\cdot (hD_x - \xi_1(x))u(x) + \widetilde{R}u,
	\end{equation}
	\begin{equation}
		\label{eq1.20}
		\text{with}\quad \widetilde{R} = {\cal O}(h): H_{\Phi_1}(\comp^n) \rightarrow L^2(\comp^n, e^{-2\Phi_1/h} L(dx)).
	\end{equation}
	
	\medskip
	\noindent
	The discussion above, developed in the case $1 < s \leq 2$, extends to the complementary range $s > 2$. Indeed, in this case, an application of \cite[Theorem 1.2]{HiLaSjZeI} yields
	\begin{multline}
		\label{eq1.21}
		a^w_{\Gamma}(x,hD_x) - \widetilde{a}^w_{\Gamma^{\Phi_1}_{h^{1/2}}}(x,hD_x) = {\cal O}(1)\, \exp\bigl(- C^{-1} h^{-\frac{1}{2s-2}}\bigr):\\
		H_{\Phi_1}(\comp^n) \rightarrow L^2(\comp^n, e^{-2\Phi_1/h} L(dx)),\quad C>0.
	\end{multline}
	Here $\Gamma^{\Phi_1}_{h^{1/2}}(x)$ is the $2n$--dimensional Lipschitz contour defined as in {\rm (\ref{eq1.6})}, {\rm (\ref{eq1.7})}, with $\omega$ replaced by $h^{1/2} \geq \omega$. We have
	\begin{equation}
		\label{eq1.22}
		\widetilde{a}^w_{\Gamma^{\Phi_1}_{h^{1/2}}}(x,hD_x) = {\cal O}(1): H_{\Phi_1}(\comp^n) \rightarrow L^2(\comp^n, e^{-2\Phi_1/h} L(dx)).
	\end{equation}
	It is then easy to see that we still get (\ref{eq1.19}) for $s>2$.
	
	\medskip
	\noindent
	Using (\ref{eq1.19}) and arguing as in~\cite{Sj90},~\cite{HeSjSt},~\cite{Rouleux}, we get the following essentially well known result, see also Proposition 4.1 of~\cite{Rouleux}.
	\begin{prop}
		\label{Toeplitz_identity}
		Let $a\in {\cal G}^s_b(\Lambda_{\Phi_0})$, $s>1$, and let $\Phi_1 \in C^{1,1}(\comp^n)$ be such that {\rm (\ref{eq1.5})} holds. Let $\widetilde{a}\in {\cal G}^s_b(\comp^{2n})$ be an almost holomorphic extension of $a$ such that ${\rm supp}\, \widetilde{a} \subset \Lambda_{\Phi_0} + B_{{\bf C}^{2n}}(0,C_0)$, for some $C_0>0$. Let $\psi \in W^{1,\infty}(\comp^n) \Longleftrightarrow \psi \in L^{\infty}(\comp^n)$, $\nabla \psi \in L^{\infty}(\comp^n)$. We have for $u,v\in H_{\Phi_1}(\comp^n)$,
		\begin{multline}
			\label{eq1.23}
			(\psi\, a^w(x,hD_x)u,v)_{H_{\Phi_1}} = \int \psi(x) \widetilde{a}(x,\xi_1(x))u(x)\overline{v(x)} e^{-2\Phi_1(x)/h}\,L(dx) \\
			+ {\cal O}(h) \norm{u}_{H_{\Phi_1}}\norm{v}_{H_{\Phi_1}}.\qquad\qquad
		\end{multline}
	\end{prop}
	
	\noindent
	Let $b\in {\cal G}^s_b(\Lambda_{\Phi_0})$, $s>1$, and let us apply Proposition \ref{Toeplitz_identity} with $b^w(x,hD_x)v$ replacing $v$, for some $v\in H_{\Phi_1}(\comp^n)$, using also \eqref{eq1.23} for $b^w(x,hD_x)$. We obtain
	\begin{multline}
		\label{eq1.24}
		\int \psi(x) \widetilde{a}(x,\xi_1(x))\overline{\widetilde{b}(x,\xi_1(x))} u(x)\overline{v(x)} e^{-2\Phi_1(x)/h}\,L(dx)
		\\ = (\psi\, a^w(x,hD_x)u,b^w(x,hD_x)v)_{H_{\Phi_1}} + {\cal O}(h) \norm{u}_{H_{\Phi_1}}\norm{v}_{H_{\Phi_1}}.\qquad
	\end{multline}
	Here $\widetilde{b}\in {\cal G}^s_b(\comp^{2n})$ is an almost holomorphic extension of $b$, as above, and we have also used the fact that $b^w(x,hD_x) = {\cal O}(1): H_{\Phi_1}(\comp^n) \rightarrow H_{\Phi_1}(\comp^n)$, see \cite{HiLaSjZeI,HiLaSjZeII}.
	
	\medskip
	\noindent
	As an application of Proposition \ref{Toeplitz_identity} and \eqref{eq1.24}, we derive an elliptic estimate for future reference. Let us make an assumption on $a\in\GG_{\rm b}^s(\Lambda_{\Phi_0})$ that there exists a bounded open subset $U\subset\comp^n$ with a constant $C>0$ such that
	\begin{equation}
		\label{ellipticity 0}
		|a(x,\xi)| \geq 2/C,\quad (x,\xi)\in \Lambda_{\Phi_0},\ \,x\in\comp^n\setminus U.
	\end{equation}
	Under this assumption we have the following:
	\begin{prop}
		\label{prop:elliptic estimate}
		Suppose $a\in\GG_{\rm b}^s(\Lambda_{\Phi_0})$ satisfies \eqref{ellipticity 0} for some bounded open set $U\subset\CC^n$ and some $C>0$. Let $\widetilde{a}\in {\cal G}^s_b(\comp^{2n})$ be an almost holomorphic extension of $a$ as in Proposition \ref{Toeplitz_identity}, and let $\Phi_1 \in C^{1,1}(\comp^n)$ be such that \eqref{eq1.5} holds. Then there exists $h_0>0$ such that for all $0<h<h_0$ and $u\in H_{\Phi_1}$ we have
		\begin{equation}
			\label{Elliptic estimate}
			\int_{\comp^n\setminus U} |u(x)|^2 e^{-\frac{2\Phi_1(x)}{h}} L(dx) \leq \OO(1)\|a^w(x,hD_x)u\|_{H_{\Phi_1}(\comp^n)}^2 + \OO(h)\|u\|_{H_{\Phi_1}(\comp^n)}^2 .
		\end{equation}
	\end{prop}
	
	\begin{proof}
		It follows from Proposition \ref{Toeplitz_identity} and more specifically \eqref{eq1.24} that for $u\in H_{\Phi_1}(\comp^n)$,
		\begin{equation}
			\label{Toep Ellip}
			\int\! \left|\widetilde{a} (x,\xi_1(x) )\right|^2 \!|u(x)|^2 e^{-\frac{2\Phi_1(x)}{h}} L(dx) \leq \|a^w(x,hD_x)u\|_{H_{\Phi_1}(\comp^n)}^2 + \OO(h)\|u\|_{H_{\Phi_1}(\comp^n)}^2 .
		\end{equation}
		Recalling from \eqref{dist Lambda Phi_0} that $\dist((x,\xi_1(x)),\Lambda_{\Phi_0}) =\OO(1)h^{1-\frac{1}{s}}$, then for $h$ sufficiently small we have, in view of \eqref{ellipticity 0} and the fact that $\widetilde{a}\in\GG_{\rm b}^s(\CC^{2n})$, $\widetilde{a}\lvert_{\Lambda_{\Phi_0}}=a$,
		\[
		\left|\widetilde{a} (x,\xi_1(x)) \right| \geq 1/C,\quad x\in\comp^n\setminus U.
		\]
		Combining this with \eqref{Toep Ellip} we obtain \eqref{Elliptic estimate}.
	\end{proof}
	
	\medskip
	\noindent
	We finish this section by discussing the composition of semiclassical Gevrey operators in the real domain. It has been proved in \cite{Lascar1988} and \cite{Lascar1997} that for $a, b\in\GG_{\rm b}^s(\RR^{2n})$ one has $a^w(x,hD)\circ b^w(x,hD) = c^w(x,hD;h)$ where $c(\cdot;h)=a\# b \in\GG_{\rm b}^s(\RR^{2n})$. An alternative proof of this result has been provided in \cite[Section 3.3]{HiLaSjZeI} using contour deformations. For future reference, we note a slightly finer characterization of the composed symbol $c=a\#b$ than $c\in \GG_{\rm b}^s(\RR^{2n})$ as follows:
	\begin{prop}
		\label{prop:composed symbol}
		Let $a, b\in\GG_{\rm b}^s(\RR^{2n})$ be $h$-independent symbols for some $s>1$, and let $c^w(x,hD;h) = a^w(x,hD)\circ b^w(x,hD)$. Then the symbol $c$ satisfies
		\begin{equation}
			\label{eq:c=ab+hr}
			c(x,\theta;h) = a(x,\theta) b(x,\theta) + h r(x,\theta;h),\quad (x,\theta)\in\RR^{2n},
		\end{equation}
		for some $r(\cdot;h)\in\GG_{\rm b}^s(\RR^{2n})$.
	\end{prop} 
	
	\begin{proof}
		Let us first recall the following oscillatory integral representation of the composed symbol $c=a\#b$, see for instance \cite[Chapter 4]{zw2012}:
		\begin{multline}
			\label{eq:c repn}
			c(x,\theta;h) = \frac{1}{(\pi h)^{2n}}\! \int_{\RR^{4n}} e^{-\frac{2i}{h}\sigma(y_1,\eta_1;y_2,\eta_2)} \\
			a(x+y_1,\theta+\eta_1) b(x+y_2,\theta+\eta_2) \,dy_1 d\eta_1 dy_2 d\eta_2.
		\end{multline}
		Here $\sigma$ is the standard symplectic form on $\RR^{2n}$. Let $\chi\in\GG_{\rm c}^s(\RR^{4n})$ be a Gevrey cutoff function such that $\chi(Y)=1$ for $|Y|\leq 1$, $Y\in\RR^{4n}$, with $\supp\chi\subset B_{\RR^{4n}}(0,2)$, and let
		\begin{multline}
			\label{eq:rchi}
			r_\chi(x,\theta;h) = 
			\frac{1}{(\pi h)^{2n}}\! \int_{\RR^{4n}} e^{-\frac{2i}{h}\sigma(y_1,\eta_1;y_2,\eta_2)} (1-\chi(y_1,\eta_1,y_2,\eta_2)) \\
			a(x+y_1,\theta+\eta_1) b(x+y_2,\theta+\eta_2) \,dy_1 d\eta_1 dy_2 d\eta_2.
		\end{multline}
		It has been established in \cite[Proposition 3.8]{HiLaSjZeI} that for some $C>0$ uniformly in $h\in(0,1]$ we have for all $\alpha, \beta\in\NN^n$ and $(x,\theta)\in \RR^{2n}$,
		\begin{equation}
			\label{eq:rchi Gevrey}
			|\partial_x^\alpha \partial_\theta^\beta r_\chi(x,\theta;h)|\leq C^{1+|\alpha|+|\beta|} \alpha!^s \beta!^s \exp \left(-\frac{1}{\OO(1)} h^{-\frac{1}{s}}\right) .
		\end{equation}
		To analyze the term $c-r_\chi$, we consider the following more general integral: 
		\[
		I_\chi (x;h) = h^{-N/2} \int_{\RR^N} e^{iq(y)/h} \chi(y) a(x+y)\,dy ,
		\]
		where $q(y) = \frac{1}{2} Ay\cdot y$ is a real non-degenerate quadratic form on $\RR^N$, $a\in \GG_{\rm b}^s(\RR^N)$ and $\chi\in\GG_{\rm c}^s(\RR^N)$ satisfies $\chi(y)=1$ for $|y|\leq 1$. 
		
		\noindent
		By Parseval's formula, writing $\tau_x a(y) =a(x+y)$ we have 
		\[
		I_\chi (x;h) = C_A \int_{\RR^N} e^{-\frac{ih}{2} A^{-1}\eta\cdot\eta}\widehat{\chi \tau_x a}(\eta)\, d\eta,\quad C_A = \frac{(2\pi)^{-N/2}e^{i\frac{\pi}{4}\sgn A}}{\abs{\det A}^{1/2}}
		\]
		where $\sgn A$ is the signature of $A$, $\widehat{\chi \tau_x a}(\eta) = \int_{\RR^N} e^{-iy\cdot\eta} \chi(y) a(x+y) \,dy$ is the Fourier transform. Using $e^{i\sigma} = 1 + i\sigma\int_0^1 e^{it\sigma} dt$, $\sigma\in\RR$, we get by $\int_{\RR^N}\widehat{u}(\eta)\,d\eta = (2\pi)^N u(0)$,
		\begin{equation}
			\label{eq:Ichi}
			I_\chi (x;h) = C_A ((2\pi)^N a(x) + h I_{\chi,1}(x;h)),
		\end{equation}
		where $\displaystyle I_{\chi,1}(x;h) = -\frac{i}{2}\int_0^1 \int_{\RR^N} e^{-\frac{ith}{2}A^{-1}\eta\cdot\eta} (A^{-1}\eta\cdot\eta) \widehat{\chi \tau_x a}(\eta)\, d\eta \,dt$.
		
		\noindent
		To derive Gevrey estimates for $I_{\chi,1}(x;h)$, we observe that for every $\alpha\in\NN^N$,
		\begin{equation}
			\label{eq:Ichi1,u_xalpha}
			\begin{gathered}
				\partial_x^\alpha I_{\chi,1}(x;h) = \frac{i}{2}\int_0^1\! \int_{\RR^N} e^{-\frac{ith}{2}A^{-1}\eta\cdot\eta} \widehat{u_{x,\alpha}}(\eta)\, d\eta \,dt, \\
				\text{with}\quad u_{x,\alpha}(y) = (A^{-1}\partial_y\cdot\partial_y) (\chi(y)\partial^\alpha a(x+y)).
			\end{gathered}
		\end{equation}
		It follows that $\abs{	\partial_x^\alpha I_{\chi,1}(x;h)}\leq \frac{1}{2}\norm{\widehat{u_{x,\alpha}}}_{L^1(\RR^N)}$. By Cauchy--Schwarz inequality,
		\[
		\int_{\RR^N}\! \abs{\widehat{u}(\eta)}\,d\eta \leq \left(\!\int_{\RR^N}\! \abs{\widehat{u}(\eta)}^2 (1+|\eta|^2)^{k} d\eta\!\right)^{1/2}\!\! \left(\!\int_{\RR^N}\! (1+|\eta|^2)^{-k} d\eta \!\right)^{1/2} 
		\!\!= C_k \|u\|_{H^k(\RR^N)}
		\]
		for any $k>N/2$, $k\in\NN$. We obtain therefore $\abs{\partial_x^\alpha I_{\chi,1}(x;h)} \leq C \|u_{x,\alpha}\|_{H^k(\RR^N)}$ for some $k\in\NN$ and $C>0$ depending only on the dimension $N$. Recalling the definition of $u_{x,\alpha}$ given in \eqref{eq:Ichi1,u_xalpha} and the assumptions that $\chi\in\GG_{\rm c}^s$, $a\in\GG_{\rm b}^s$, we conclude that there exists $C>0$ uniformly in $h\in(0,1]$ such that
		\begin{equation}
			\label{eq:Ichi1 Gevrey}
			\abs{\partial_x^\alpha I_{\chi,1}(x;h)} \leq C^{1+|\alpha|} \alpha!^s ,\quad \forall \alpha\in\NN^N,\ x\in\RR^N .
		\end{equation} 
		Combining \eqref{eq:Ichi}, \eqref{eq:Ichi1 Gevrey} and applying the result to the integral representation of $c-r_\chi$, we obtain via a direct computation that 
		\[
		c(x,\theta;h)-r_\chi(x,\theta;h) = a(x,\theta) b(x,\theta) + h r_1(x,\theta;h),\quad(x,\theta)\in\RR^{2n},
		\]
		for some $r_1(\cdot;h)\in\GG_{\rm b}^s(\RR^{2n})$. This and \eqref{eq:rchi Gevrey}, together with equations \eqref{eq:c repn}, \eqref{eq:rchi}, imply the desired result \eqref{eq:c=ab+hr}.
	\end{proof}
	
	\section{Complex deformations and Proof of Theorem \ref{thm:special case}}
	\label{sec:complex deformation}
	
	Let $p\in\GG_{\rm b}^s(\RR^{2n})$, $s>1$. Let $z_0\in\partial\Sigma(p)\setminus\Sigma_\infty(p)$ satisfy conditions \eqref{special case}--\eqref{nontrapping special case}. We first recall from \cite[Lemma 4.2]{DeSjZw} the existence of an escape function:
	\begin{lemma}
		\label{lem:escape function special case}
		Let $p$ and $z_0$ be given as in Theorem \ref{thm:special case} such that \eqref{special case}--\eqref{nontrapping special case} hold. Then there exists $G\in C_{\rm c}^\infty (\RR^{2n};\RR)$ such that for some constant $c>0$ we have 
		\begin{equation}
			\label{escape G}
			H_{\Im p} G(\rho) < -c< 0,\quad \rho\in p^{-1}(z_0) .
		\end{equation}
	\end{lemma}
	
	\noindent
	We now introduce small compactly supported deformations of the real space $\RR^{2n}$. Let $H_G(\rho) = (G_\xi'(\rho),-G_x'(\rho))$ be the Hamilton vector field of $G$, and set
	\begin{equation}
		\label{Lambda tG}
		\Lambda_{tG} := \{\rho + it H_{G}(\rho) ; \rho\in T^* \RR^n\} \subset \CC^{2n},\quad t\in\RR,\ \, |t|\ \text{small}.
	\end{equation}
	We note that $\Lambda_{tG}$ is I-Lagrangian and R-symplectic, in the sense that $\sigma\lvert_{\Lambda_{tG}}$ is real and non-degenerate, where $\sigma=d\xi\wedge dx$ is the complex symplectic form on $\comp_x^n\times\comp_\xi^n$. Let $\ptilde\in\GG_b^s(\CC^{2n})$ be an almost holomorphic extension of $p$, which is supported in a bounded tubular neighborhood of $\RR^{2n} \subset \CC^{2n}$. We shall explore the behavior of $\ptilde\lvert_{\Lambda_{tG}}$ near some $\rho\in {\rm neigh}\,(p^{-1}(z_0);\RR^{2n})$. 
	
	\medskip
	\noindent
	For simplicity, we may assume $z_0=0$ by subtracting $z_0$ from $p$. Let us use Taylor's formula to expand $\ptilde$ at $\rho\in {\rm neigh}\,(p^{-1}(0);\RR^{2n})$ and write 
	\[
	\begin{split}
		\ptilde(\rho+itH_G(\rho)) &= p(\rho) + it H_G \ptilde(\rho) + \OO(t^2) \\ 
		&= p(\rho) - it H_{\Re p}G(\rho) + t H_{\Im p}G(\rho) + \OO(t^2),
	\end{split}
	\]
	where we used the facts that $\ptilde\lvert_{\RR^{2n}}=p$, $\overline{\partial}\ptilde\, \lvert_{\RR^{2n}}=0$. We obtain therefore
	\[
	\Re \ptilde(\rho+itH_G(\rho)) = \Re p(\rho) + t H_{\Im p}G(\rho) + \OO(t^2).
	\]
	Letting $\Omega\subset\RR^{2n}$ be a sufficiently small neighborhood of $p^{-1}(0)$ which is compact in $\RR^{2n}$ since $0\notin \Sigma_\infty(p)$, we can infer from \eqref{special case}, \eqref{escape G} and the computation above that there exist $\gamma>0$ and $t_0>0$ such that
	\begin{equation}
		\label{Re p on Lambda tG}
		\Re \ptilde(\rho+itH_G(\rho)) \geq \gamma|t|,\quad \rho\in\Omega,\ \,-t_0 < t \leq 0.
	\end{equation}
	
	\noindent
	Let us now move to the FBI transform side. Let $\phi(x,y)$ be a holomorphic quadratic form on $\CC_x^n\times\CC_y^n$ with $\Im \phi_{yy}'' >0$, $\det \phi_{xy}''\neq 0$. To $\phi$ we associate the complex linear canonical transformation
	\begin{equation}
		\label{canonical trans}
		\kappa_\phi :\CC^{2n}\owns (y,-\phi_y'(x,y))\mapsto (x,\phi_x'(x,y))\in \CC^{2n}.
	\end{equation}
	Recalling for instance \cite[Theorem 13.5]{zw2012} we have
	\[
	\kappa_\phi(\RR^{2n}) = \Lambda_{\Phi_0},\quad \Lambda_{\Phi_0}:=\left\{ \left( x,\frac{2}{i}\frac{\partial\Phi_0}{\partial x} (x) \right) ; x\in\CC^n \right\} \subset \CC^{2n},
	\]
	where
	\begin{equation}
		\label{eq:Phi0}
		\Phi_0(x) = \max_{y\in\RR^n}\,-\Im\phi(x,y)
	\end{equation}
	is a strictly plurisubharmonic quadratic form on $\CC^n$. 
	
	\noindent
	Let us also recall from \cite[Section 2]{Sj10} that
	\begin{equation}
		\label{Lambda Phi_t}
		\kappa_\phi(\Lambda_{tG}) = \Lambda_{\Phi_t} = \left\{ \left( x,\frac{2}{i}\frac{\partial\Phi_t}{\partial x} (x) \right) ; x\in\CC^n \right\} \subset \CC^{2n},
	\end{equation}
	\begin{equation}
		\label{Phi_t critical value}
		\text{with}\quad\Phi_t(x) = \text{v.c.}_{(y,\eta)\in\CC^n\times\RR^n} (-\Im\phi(x,y)-\eta\cdot\Im y + tG(\Re y,\eta)).
	\end{equation}
	Computing the critical value in $\eqref{Phi_t critical value}$ as a perturbation of $t=0$, we obtain 
	\begin{equation}
		\label{Phi_t perturbs Phi_0}
		\Phi_t(x) = \Phi_0(x) + tG \left( \kappa_\phi^{-1} \left( x,\frac{2}{i}\frac{\partial\Phi_0}{\partial x} (x) \right) \right) + \OO(t^2).
	\end{equation}
	To see this, we first observe by comparing \eqref{eq:Phi0} and \eqref{Phi_t critical value} that $\Phi_t(x)\lvert_{t=0} = \Phi_0(x)$ since
	\[
	\text{v.c.}_{(y,\eta)\in\CC^n\times\RR^n} (-\Im\phi(x,y)-\eta\cdot\Im y) = \max_{y\in\RR^n}\,-\Im\phi(x,y),
	\]
	and that there is a single critical point $(y_0,\eta_0) = \kappa_\phi^{-1} \bigl( x,\frac{2}{i}\frac{\partial\Phi_0}{\partial x} (x) \bigr)\in\RR^{2n}$ due to the non-degeneracy of $\Im \phi_{yy}''$. It follows that the critical value in \eqref{Phi_t critical value} is also evaluated at a single critical point when $t$ is small. Let us then differentiate the critical value $\Phi_t(x)$ in the parameter $t$, in view of \eqref{Phi_t critical value}, we get 
	\[
	(\partial_t \Phi_t(x))\lvert_{t=0} = G(\Re y_0,\eta_0) = G\left( \kappa_\phi^{-1} \left( x,\frac{2}{i}\frac{\partial\Phi_0}{\partial x} (x) \right) \right),
	\]
	where we noticed that $y_0\in\RR^n$. We obtain therefore \eqref{Phi_t perturbs Phi_0} by Taylor's formula.
	
	\noindent
	Let us introduce the FBI-Bargmann transform associated to $\phi$:
	\begin{equation}
		\label{T FBI defn}
		Tu(x;h) = C_n h^{-3n/4} \int_{\RR^n} e^{i\phi(x,y)/h} u(y) dy,
	\end{equation}
	where $C_n$ is chosen such that $T: L^2(\RR^n) \to H_{\Phi_0}(\CC^n)$ is unitary, see \cite[Theorem 13.7]{zw2012}. Let $a = p\circ \kappa_\phi^{-1}\in\GG_b^s(\Lambda_{\Phi_0})$, then $\widetilde{a} := \widetilde{p}\circ\kappa_\phi^{-1} \in \GG_b^s(\CC^{2n})$ is an almost holomorphic extension of $a$ such that $\supp \,\widetilde{a}\subset \Lambda_{\Phi_0} + B_{\comp^{2n}}(0,C_0)$ for some $C_0>0$. The exact Egorov theorem (see for example \cite[Theorem 13.9]{zw2012}) implies that
	\[
	a^w(x,hD_x)\circ T = T\circ p^w (x,hD) ,
	\]
	where $a^w(x,hD_x)$ is the semiclassical Weyl quantization of $a\in\GG_{\rm b}^s(\Lambda_{\Phi_0})$ given by \eqref{eq1.3} while $p^w(x,hD)$ is the semiclassical Weyl quantization of $p\in\GG_{\rm b}^s(\RR^{2n})$ defined in \eqref{eq:Weyl quantization}.
	
	\medskip
	\noindent
	Let us set $t = - \epsilon h^{1-\frac{1}{s}}$ for $\epsilon>0$ small but independent of $h$. In view of \eqref{Phi_t perturbs Phi_0}, we see that \eqref{eq1.5} holds with $\Phi_t$ replacing $\Phi_1$, if $\epsilon>0$ is small enough. Applying \eqref{eq1.23} with $\psi\equiv 1$ and the weight $\Phi_t$, we get
	\begin{equation}
		\label{Toeplitz 1}
		\begin{split}
			\Re( a^w(x,hD_x) u, u)_{H_{\Phi_t}} &= \int  \Re\widetilde{a}\left( x, \frac{2}{i}\frac{\partial\Phi_t}{\partial x} (x)\right) |u(x)|^2 e^{-2\Phi_t(x)/h}\, L(dx) \\
			& \quad\quad +\OO(h) \|u\|_{H_{\Phi_t}}^2.
		\end{split}
	\end{equation}
	Let $\Omega\subset\RR^{2n}$ be an open neighborhood of $p^{-1}(0)$ such that \eqref{Re p on Lambda tG} holds. We set
	\begin{equation}
		\label{def:U}
		U = \pi_x(\kappa_\phi(\{ \rho + it H_G(\rho) : \rho \in \Omega \}))\subset\CC^n,
	\end{equation}
	where $\pi_x : \Lambda_{\Phi_t}\owns(x,\xi)\mapsto x\in\CC^n$ is the projection map, and $\kappa_\phi$ is the canonical transform defined in \eqref{canonical trans}. We note that $U$ is open and bounded since $p^{-1}(0)\subset\RR^{2n}$ is compact, which is due to the assumption $0\notin \Sigma_\infty(p)$. Combing \eqref{Re p on Lambda tG} with \eqref{Lambda tG}, \eqref{Lambda Phi_t}, and recalling $\widetilde{a} = \widetilde{p}\circ\kappa_\phi^{-1}$, we obtain for some $\gamma>0$, $t_0>0$ as in \eqref{Re p on Lambda tG},
	\begin{equation}
		\label{Re atilde on U}
		\Re \widetilde{a}\left(x,\frac{2}{i}\frac{\partial\Phi_t}{\partial x}(x)\right) \geq \gamma|t|,\quad x\in U,\ \,-t_0 < t \leq 0.
	\end{equation}
	We shall next show that $a=p\circ\kappa_\phi^{-1}$ satisfies the condition \eqref{ellipticity 0} with $U$ given in \eqref{def:U}. To this end, we first note that
	\begin{equation*}
		K_0 := \{x\in\CC^n : a(x,\xi)=0,\ (x,\xi)\in\Lambda_{\Phi_0} \} = \pi_x(\kappa_\phi(p^{-1}(0))) \subset\CC^n\ \,\text{is compact}.
	\end{equation*}
	Since $\pi_x : \Lambda_{\Phi_t}\to\CC^n$, $-t_0<t\leq 0$, $\kappa_\phi:\RR^{2n}\to\Lambda_{\Phi_0}$ or $\Lambda_{tG}\to\Lambda_{\Phi_t}$, and $g_t : \RR^{2n}\owns\rho\mapsto \rho+itH_G(\rho)\in\Lambda_{tG}$ are all diffeomorphisms between corresponding submanifolds of $\CC^{2n}$, we get that $U_0:=\pi_x(\kappa_\phi(\Omega))\subset\CC^n$ is an open neighborhood of $K_0$ and that $U$ defined in \eqref{def:U} is open. To see $K_0\subset U$, it suffices to show that $\dist(\partial U , K_0)>0$ (note that $U$ is bounded). Let $x_t\in\partial U$, $x_t = \pi_x(\kappa_\phi(\rho+itH_G(\rho)))$ for some $\rho\in\partial\Omega$. Since $x=\pi_x(\kappa_\phi(\rho))\in\partial U_0$ and $K_0\subset U_0$, we have $\dist(x , K_0)\geq \dist(\partial U_0 , K_0) =: \delta_0 > 0$. Noting that $|x_t-x| = \OO(|t|)$, we get therefore $\dist(x_t , K_0)\geq \delta_0/2$ for all $-t_0<t\leq 0$ by taking a smaller $t_0$ if necessary. We conclude from the above discussion that 
	\begin{equation}
		\label{a outside U}	
		|a(x,\xi)| \geq 2/C,\quad(x,\xi)\in\Lambda_{\Phi_0},\ x\in\CC^n\setminus U,\quad -t_0<t\leq 0, 
	\end{equation}
	for some $C>0$ uniform in $-t_0<t\leq 0$ (noting that $U$ depends on $t$). 
	
	\medskip
	\noindent
	We now take $t=-\epsilon h^{1-\frac{1}{s}}$ with $\epsilon>0$ small but fixed. Combining \eqref{Toeplitz 1}, \eqref{Re atilde on U} we get
	\[
	\begin{split}
		\Re( a^w(x,hD_x) u, u)_{H_{\Phi_t}} \geq\, &{ } \gamma\epsilon h^{1-\frac{1}{s}} \int_U |u(x)|^2 e^{-2\Phi_t(x)/h}\, L(dx) - \OO(h) \|u\|_{H_{\Phi_t}}^2 \\
		&- \OO(1)\int_{\comp^n\setminus U} |u(x)|^2 e^{-2\Phi_t(x)/h}\, L(dx).
	\end{split}	
	\]
	In view of \eqref{a outside U} we apply Proposition \ref{prop:elliptic estimate} with $\Phi_t$ in place of $\Phi_1$, we obtain therefore
	\[
	\begin{gathered}
		\Re ( a^w(x,hD_x) u, u)_{H_{\Phi_t}} \geq (\gamma\epsilon h^{1-\frac{1}{s}}-\OO(h))\|u\|_{H_{\Phi_t}}^2 - \OO(1)\int_{\comp^n\setminus U} |u(x)|^2 e^{-2\Phi_t(x)/h} L(dx) \\
		\qquad\qquad \geq (\gamma\epsilon h^{1-\frac{1}{s}}-\OO(h))\|u\|_{H_{\Phi_t}}^2 - \OO(1)\|a^w(x,hD_x) u\|_{H_{\Phi_t}}^2 - \OO(h)\|u\|_{H_{\Phi_t}}^2 \\
		\geq \delta h^{1-\frac{1}{s}} \|u\|_{H_{\Phi_t}}^2 - \OO(1)\|a^w(x,hD_x) u\|_{H_{\Phi_t}}^2 ,\quad 0<h<h_0.\quad
	\end{gathered}
	\] 
	Here $h_0,\,\delta>0$ are small constants. By Peter--Paul inequality,
	\[
	\Re ( a^w(x,hD_x) u, u)_{H_{\Phi_t}} \leq \frac{\delta}{2}h^{1-\frac{1}{s}} \|u\|_{H_{\Phi_t}}^2 + \frac{1}{2\delta}h^{\frac{1}{s}-1} \|a^w(x,hD_x) u\|_{H_{\Phi_t}}^2 .
	\]
	Combining the estimates above, we conclude that for some $C>0$,
	\begin{equation}
		\label{zero not in Spec}
		\|a^w(x,hD_x) u\|_{H_{\Phi_t}} \geq C^{-1} h^{1-\frac{1}{s}} \|u\|_{H_{\Phi_t}},\quad u\in H_{\Phi_t},\  0<h<h_0 .
	\end{equation}
	Writing $P(h) = p^w(x,hD) + h p_1^w(x,hD;h)$ for some $p_1(\cdot;h)\in \GG_{\rm b}^s(\RR^{2n})$, we get 
	\[
	T  P(h)  T^{-1} = a^w(x,hD_x) + h a_1^w(x,hD_x;h)
	\] 
	with $a_1(\cdot;h)\in\GG_{\rm b}^s(\Lambda_{\Phi_0})$. We can infer from the proofs in \cite{HiLaSjZeI} that 
	\[
	a_1(\cdot;h)\in\GG_{\rm b}^s(\Lambda_{\Phi_0}) \implies a_1^w(x,hD_x;h) = \OO(1): H_{\Phi_t}\to H_{\Phi_t} .
	\]
	It follows that there exists $C>0$ uniformly in $0<h<h_0$ such that
	\begin{equation}
		\label{estimate P(h)-z}
		\|T (P(h) - z)T^{-1} u\|_{H_{\Phi_t}} \geq (2C)^{-1} h^{1-\frac{1}{s}}\|u\|_{H_{\Phi_t}},\quad u\in H_{\Phi_t},\ |z|\leq (2C)^{-1}h^{1-\frac{1}{s}}.
	\end{equation} 
	This estimate implies that $P(h)-z: L^2(\RR^n)\to L^2(\RR^n)$ is injective, as we note that $H_{\Phi_t}=H_{\Phi_0}$ as linear spaces and that $T:L^2(\RR^n) \to H_{\Phi_0}$ is unitary. Since $0\notin \Sigma_\infty(p)$, it has been shown in the proof of \cite[Proposition 3.3]{DeSjZw} that $P(h)-z:L^2\to L^2$ is a Fredholm operator with index $0$ for $z$ in an $\OO(1)$--neighborhood of $0$ and $h$ sufficiently small. We can therefore conclude that for some $C>0$ and $h_0>0$, 
	\[
	\left\{ z\in\comp : |z|< C^{-1} h^{1-\frac{1}{s}} \right\} \cap \sigma(P(h)) = \emptyset,\quad 0<h<h_0.
	\]
	Recalling $\|\Phi_t-\Phi_0\|_{L^\infty(\comp^n)} \leq C^{-1} h^{1-\frac{1}{s}}$ and \eqref{eq1.2}, we deduce from \eqref{estimate P(h)-z} that
	\[
	\|T (P(h) - z)T^{-1} u\|_{H_{\Phi_0}} \geq C^{-1} h^{1-\frac{1}{s}} e^{-\OO(1)h^{-1/s} }\|u\|_{H_{\Phi_0}},\quad u\in H_{\Phi_0},\ |z|\leq C^{-1}h^{1-\frac{1}{s}}.
	\]
	Combining this with the invertibility of $P(h)-z$ and the unitarity of $T$, we get 
	\[
	(P(h) - z)^{-1} = \OO(1) e^{\OO(1)h^{-1/s} } : L^2(\RR^n)\to L^2(\RR^n),\quad |z|\leq C^{-1}h^{1-\frac{1}{s}},
	\]
	for $0<h<h_0$. This completes the proof of Theorem \ref{thm:special case}.
	
	\section{Proof of Theorem \ref{main theorem}}
	\label{sec: prove Thm 1}
	
	We start with an adaption of \cite[Lemma 4.1]{DeSjZw} to the Gevrey class, which can be achieved by following the proof in \cite{DeSjZw} and using a division theorem for Gevrey functions given in \cite{Br}. In fact, we only need a special case of \cite[Theorem 5]{Br} (when $m=1$):
	\begin{prop}
		\label{prop:division thm}
		Let $\Omega$ be a domain in $\RR^d$ and let $f\in\GG^s(\RR\times\Omega;\comp)$. Let $P(t,x)=t+a(x)$ with $a\in\GG^s(\Omega;\comp)$. Then there exist $Q\in\GG^s(\RR\times\Omega;\comp)$ and $R\in\GG^s(\Omega;\comp)$ such that
		\begin{equation}
			f(t,x) = Q(t,x) P(t,x) + R(x),\quad (t,x)\in\RR\times\Omega .
		\end{equation}
	\end{prop}
	
	\noindent
	Following an argument in \cite{Nirenberg}, we can deduce a preparation theorem for Gevrey functions from the division theorem above.
	
	\begin{prop}
		\label{prop:preparation thm}
		Let $f(t,x)\in\GG^s(\RR\times\RR^d;\comp)$. Suppose that $f(0,0)=0$, $\frac{\partial f}{\partial t}(0,0)\neq 0$. Then in an open neighborhood of $(0,0)\in \RR\times\RR^d$ we have the factorization
		\begin{equation}
			f(t,x) = q(t,x) (t+\lambda(x))
		\end{equation}
		where $q$ and $\lambda$ are $\GG^s$ complex-valued functions with $q(0,0)\neq 0$, $\lambda(0)=0$.
	\end{prop}
	
	\begin{proof}
		Let us introduce a generic complex variable $\lambda\in\CC$. We note that $F(t,x,\lambda):=f(t,x)\in\GG^s(\RR_t\times\RR_x^d\times\CC_\lambda)$, $P(t,x,\lambda) := t + \lambda$, then by Proposition \ref{prop:division thm}, we have
		\begin{equation}
			\label{eq:f,Q,R}
			f(t,x) = Q(t,x,\lambda) (t+\lambda) + R(x,\lambda),
		\end{equation}
		for some $Q\in\GG^s(\RR_t\times\RR_x^d\times\CC_\lambda)$, $R\in \GG^s(\RR_x^d\times\CC_\lambda)$. Since $f(0,0)=0$ we must have $R(0,0)=0$, thus $\frac{\partial f}{\partial t}(0,0)\neq 0$ implies $Q(0,0,0)\neq 0$. Moreover, applying $\partial_{\overline{\lambda}}$ to \eqref{eq:f,Q,R} and taking $(t,x,\lambda)=(0,0,0)$, we get $\frac{\partial R}{\partial \overline{\lambda}}(0,0) = 0$, thus $\frac{\partial \overline{R}}{\partial\lambda}(0,0)=0$. Suppose $\frac{\partial R}{\partial\lambda}(0,0) = 0$, we deduce from \eqref{eq:f,Q,R} by applying $\partial_\lambda$ and taking $(t,x,\lambda)=(0,0,0)$ that $Q(0,0,0)=0$, a contradiction. Therefore, we have $\frac{\partial R}{\partial\lambda}(0,0) \neq 0$. We conclude from the discussion above that the Jacobian matrix $\frac{\partial (R,\overline{R})}{\partial(\lambda,\overline{\lambda})}(0,0)\neq 0$. It 
		follows from the implicit function theorem in the Gevrey-$s$ class by \cite{Komatsu} that there exists an open neighborhood $U\subset\RR^d$ of $0$ and $\lambda(x)\in\GG^s(U ; \CC)$ such that $\lambda(0)=0$, and $R(x, \lambda(x))=0$ for all $x\in U$. The desired factorization follows by letting $q(t,x):= Q(t,x,\lambda(x))\in\GG^s (\RR_t\times U_x)$.
	\end{proof}
	
	Following the proof of \cite[Lemma 4.1]{DeSjZw} while replacing the Malgrange preparation theorem used there by Proposition \ref{prop:preparation thm}, and using a $\GG^s$ (Gevrey-$s$) partition of unity, we obtain the following result:
	\begin{lemma}
		\label{Lemma 4.1 DSZ}
		Let $p\in\GG_b^s(\RR^{2n})$ and $z_0\in \partial\Sigma(p)$ such that \eqref{principal type condition} and \eqref{dynamical condition} hold. Then there exists $q\in\GG_b^s(\RR^{2n})$ such that $q\neq 0$ on $p^{-1}(z_0)$ and 
		\begin{equation}
			\label{property of q}
			\Im(q(p-z_0))\geq 0\ \text{near}\ p^{-1}(z_0),\quad |d\Re (q(p-z_0))|\geq c > 0\ \text{on}\ p^{-1}(z_0).
		\end{equation}
	\end{lemma}

	We are now ready to prove Theorem \ref{main theorem}. Let $p\in\GG^s_b(\RR^{2n})$, $s>1$, and let $p^w(x,hD)$ be the principal part of the semiclassical Gevrey operator $P(h)$. Suppose $z_0\in \partial\Sigma(p)\setminus \Sigma_\infty(p)$ satisfies conditions \eqref{principal type condition} and \eqref{dynamical condition}. Lemma \ref{Lemma 4.1 DSZ} shows that there exists $q\in\GG_b^s(\RR^{2n})$ such that $q\neq 0$ on $p^{-1}(z_0)$ and \eqref{property of q} holds. Using a $\GG^s$ partition of unity, we may assume further that there exists a constant $\gamma>0$ such that
	\begin{equation}
		\label{q nonvanishing}
		|q(x,\xi)|\geq \gamma>0,\quad (x,\xi)\in \RR^{2n}.
	\end{equation}
	This is due to the compactness of $p^{-1}(z_0)$. 
	
	\noindent
	Fixing $z_0$ as above, let us consider the semiclassical Gevrey operator 
	\begin{equation}
		\label{operato P_1}
		P_1(h) := i^{-1}q^w(x,hD) (P(h)-z_0)
	\end{equation}
	whose principal part is $p_1^w(x,hD)$ with $p_1 := i^{-1}q(p-z_0)$ by Proposition \ref{prop:composed symbol}. We next verify that $p_1$ satisfies conditions \eqref{special case}--\eqref{nontrapping special case}. Noting that $p_1^{-1}(0)= p^{-1}(z_0)$ by \eqref{q nonvanishing}, it then follows from \eqref{property of q} that 
	\begin{equation}
		\label{property p_1}
		\Re p_1 \geq 0\quad\text{near }\,p_1^{-1}(0);\qquad d\,\Im p_1 \neq 0\quad\text{on }\,p_1^{-1}(0) .
	\end{equation}
	Let $\rho_0\in p_1^{-1}(0) = p^{-1}(z_0)$. By a direct computation we see that one can obtain a trajectory of $H_{\Re (e^{-i\theta(\rho_0)} p)}$ passing through $\rho_0$ that is contained in $p^{-1}(z_0)$ from a trajectory of $H_{\Im p_1}$ passing through $\rho_0$ that is contained in $p_1^{-1}(0)$ simply by reparametrization. Therefore, the condition \ref{nontrapping special case} must hold for $p_1$ and $0$, i.e. 
	\begin{equation}
		\label{nontrapping p_1}
		\begin{gathered}
			\forall \rho\in p_1^{-1}(0),\quad \textrm{the maximal trajectory of $H_{\Im p_1}$ passing through $\rho$} \\
			\textrm{contains a point where $\Re p_1 > 0$} .
		\end{gathered}
	\end{equation}
	Otherwise, the condition \eqref{dynamical condition} on $p$ and $z_0$ would be contradicted. In view of \eqref{property p_1} and \eqref{nontrapping p_1}, repeating the steps to derive \eqref{zero not in Spec}, we can conclude that there exist $h_0>0$ and $C>0$ such that (with the unitary operator $T$ defined in \eqref{T FBI defn})
	\begin{equation}
		\label{p_1 estimate}
		\|T  p_1^w(x,hD) T^{-1} u\|_{H_{\Phi_t}} \geq C^{-1} h^{1-\frac{1}{s}} \|u\|_{H_{\Phi_t}},\quad u\in H_{\Phi_t},\  0<h<h_0 .
	\end{equation}
	By Proposition \ref{prop:composed symbol} we have for some $r(\cdot;h)\in\GG_{\rm b}^s(\RR^{2n})$, 
	\[
	p_1^w(x,hD) = i^{-1} q^w(x,hD) (p^w(x,hD) - z_0) + h r^w(x,hD;h).
	\]
	Noting that $q, r(\cdot;h)\in\GG_{\rm b}^s(\RR^{2n})\implies T q^w(x,hD) T^{-1}, T r^w(x,hD;h) T^{-1} = \OO(1): H_{\Phi_t}\to H_{\Phi_t}$, we deduce from \eqref{p_1 estimate} that for a larger $C>0$ we have
	\[
	\|T  (p^w(x,hD) - z_0)  T^{-1} u\|_{H_{\Phi_t}} \geq C^{-1} h^{1-\frac{1}{s}} \|u\|_{H_{\Phi_t}},\quad u\in H_{\Phi_t},\  0<h<h_0 .
	\]
	Arguing as in the proof of Theorem \ref{thm:special case}, we conclude from the estimate above that
	\begin{equation*}
		\left\{z\in\comp : |z-z_0| < C^{-1} h^{1-\frac{1}{s}} \right\} \cap \sigma(P(h)) = \emptyset,\quad 0<h<h_0 ,
	\end{equation*}
	and that the resolvent estimate \eqref{resolvent estimate} holds for $z$ in the spectrum free region above.

\end{document}